\documentclass[a4paper,12pt]{amsart}
\usepackage[T1]{fontenc}
\usepackage[utf8]{inputenc}
\usepackage[margin=2.5cm]{geometry}

\usepackage{amsmath, amssymb,amsfonts}
\usepackage{thm-restate}
\usepackage{array}
\usepackage{multirow,todonotes}
\usepackage{tikz}
\usetikzlibrary{shapes.misc}
\usetikzlibrary{shapes.geometric}
\usetikzlibrary {decorations.pathmorphing}
\usetikzlibrary{positioning, 
                quotes}

\tikzset{unseulcoin fill/.style={append after command={
   \pgfextra
        \draw[sharp corners, fill=#1, color=#1, line width = 0mm]%
    (\tikzlastnode.west)%
    [rounded corners=0pt] |- (\tikzlastnode.north)%
    [rounded corners=0pt] -| (\tikzlastnode.east)%
    [rounded corners=0pt] |- (\tikzlastnode.south)%
    [rounded corners=5pt] -| (\tikzlastnode.west);
   \endpgfextra}}}

\tikzset{black node/.style = {draw=black,fill=black,circle, inner sep=0, minimum size=0.15cm}}
\usepackage{multicol}

\usepackage{hyperref}
\hypersetup{
  pdftitle = {Pushing Cops and Robber on Graphs of Maximum Degree 4},
  pdfauthor = {H. Gahlawat},
  colorlinks = true,
  linkcolor = black!30!red,
  citecolor = black!30!green
}
  \usepackage[capitalise, compress, nameinlink, noabbrev]{cleveref}
\crefname{rule}{}{}
\creflabelformat{rule}{#2(R#1)#3}

\usepackage{paralist}

\newtheorem{theorem}{Theorem}[section]
\newtheorem{corollary}[theorem]{Corollary}

\newtheorem{proposition}[theorem]{Proposition}
\newtheorem{lemma}[theorem]{Lemma}

\newtheorem{remark}[theorem]{Remark}

\newtheorem{claim}[theorem]{Claim}

\def\cqedsymbol{\ifmmode$\lrcorner$\else{\unskip\nobreak\hfil
\penalty50\hskip1em\null\nobreak\hfil$\lrcorner$
\parfillskip=0pt\finalhyphendemerits=0\endgraf}\fi}

\newcommand{\set}[1]{\left\{#1\right\}}

\newcommand{\R}{\mathcal{R}}
\newcommand{\C}{\mathcal{C}}

\newcommand{\csp}[1]{{\sf c_{sp}}(\overrightarrow{#1})}

\newenvironment{proofofclaim}{%
  \proof}{\endproof}

\usepackage{xspace,caption,subcaption}

\newcommand{\CR}{\textsc{Cops and Robber}\xspace}

    \usepackage{tabularx}

\let\leq\leqslant

\title[Pushing Cops and Robber]{Pushing Cops and Robber on Graphs of Maximum Degree 4}

\author[H.~Gahlawat]{Harmender Gahlawat}
\address[H.~Gahlawat]{Universit\'e Clermont Auvergne, CNRS, Clermont Auvergne INP, Mines Saint-\'Etienne, LIMOS, Clermont-Ferrand, 63000, France}
\email{harmendergahlawat@gmail.com}

\date{\today}

\begin{document}

\begin{abstract}
\textsc{Cops and Robber} is a game played on graphs where a set of \textit{cops} aim to \textit{capture} the position of a single \textit{robber}. The main parameter of interest in this game is the \textit{cop number}, which is the minimum number of cops that are sufficient to guarantee the capture of the robber. 

In a directed graph $\overrightarrow{G}$, the \textit{push} operation on a vertex $v$ reverses the orientation of all arcs incident to $v$. We consider a variation of the classical \textsc{Cops and Robber} on oriented graphs, where in its turn, each cop can either move to an out-neighbor of its current vertex or push some vertex of the graph, whereas, the robber can move to an out-neighbor in its turn. [Das et al., CALDAM, 2023] introduced this variant and established that if $\overrightarrow{G}$ is an orientation of a subcubic graph, then one cop with push ability has a winning strategy. We extend these results to establish that if $\overrightarrow{G}$ is an orientation of a $3$-degenerate graph, or of a graph with maximum degree $4$, then one cop with push ability has a winning strategy. Moreover, we establish that if $\overrightarrow{G}$ can be made to be a directed acyclic graph, then one cop with push ability has a winning strategy. 
\end{abstract}
\maketitle
\section{Introduction}\label{S:intro}
\CR is a well-studied pursuit-evasion game, where a set of cops pursue a single robber.  
We study a variant of \CR on oriented graphs. Classically, the game in the oriented setting has the following rules. The game starts with the cops placing themselves on the vertices of an oriented graph $\overrightarrow{G}$, and multiple cops may simultaneously occupy the same vertex of the graph. Then the robber chooses a vertex to start. Now the cops and the robber make alternating moves, beginning with the cops. In a cop move, each cop can either stay on the same vertex or move to a vertex in its out-neighborhood. In the robber move, the robber does the same. If at some point in the game, one of the cops occupies the same vertex as the robber, we call it the \textit{capture}. The cops win if they can capture the robber in a finite number of rounds, and if the robber can evade the capture forever, then the robber wins.

\smallskip
The \textit{cop number} $c(\overrightarrow{G})$ of an oriented graph $\overrightarrow{G}$ is the minimum number of cops needed by the Cop Player to have a winning strategy. 
We say that $\overrightarrow{G}$  is \textit{cop-win} if one cop has a winning strategy in $\overrightarrow{G}$. Most research in oriented (or directed) graphs considers the model defined above. However, there has been some research concerning variations of the game in oriented graphs~\cite{das2021cops}.

Let $\overrightarrow{uv}$ be an arc of an oriented\footnote{An oriented graph is a directed graph without self-loops and 2-cycles.} graph $\overrightarrow{G}$. 
We say that $u$ is an \textit{in-neighbor} of $v$ and $v$ is an \textit{out-neighbor} of $u$.
Let $N^-(u)$ and $N^+(u)$ denote the set of in-neighbors and out-neighbors of $u$, respectively.  Moreover, let $N^+[v] = N^+(v) \cup \{v\}$ and $N^-[v] = N^-(v) \cup \{v\}$.
A vertex without any in-neighbor is a \textit{source}, and a vertex without any out-neighbor is a \textit{sink}. A vertex $v$ is said to be \textit{dominating} if $N^+[v] = V(\overrightarrow{G})$. For a vertex $v$, the \textit{push operation on $v$}, denoted by $push(v)$, reverses the orientation of each arc incident to $v$.  
We remark that the push operation is a well-studied modification operation on directed or oriented graphs~\cite{fisher-push, klostermeyer1, klostermeyer2, Mosesian, Pretzel1, Pretzel2, Pretzel3, Garyandwood}. In this work, for convenience and for the sake of better readability, we retain the name of an oriented graph even after some vertices have been pushed, allowing a slight abuse of notation. However, there is no scope of confusion to the best of our knowledge.

~\cite{das2023cops} initiated the study of \CR on oriented graphs with respect to the push operation where the players can have the ability to push the vertices of the graph.  They defined two kinds of push ability.

\begin{enumerate}

    \item \textit{Weak push}: Let $A$ be an agent (cop/robber) having the weak push ability, and let $A$ be on a vertex $v$. Then in its turn, $A$ can either move to a vertex $u \in N^+[v]$ or can push the vertex $v$.
    
    \item \textit{Strong push}: Let $A$ be an agent (cop/robber) having the strong push ability, and let $A$ be on a vertex $v$. Then in its turn, $A$ can either move to a vertex $u \in N^+[v]$ or can push any vertex of the graph. 
\end{enumerate}

We are interested in graph classes which become cop-win when the cops have ability to push the vertices but have higher cop number otherwise. It is a straightforward observation that unless $\overrightarrow{G}$ has a dominating vertex, one cop even with the strong push ability cannot win in $\overrightarrow{G}$ against a robber that has weak push ability. To win against a single cop, the robber will push the position of the cop whenever the cop is on an in-neighbor of the robber's current position. 
Hence, in this work, we restrict our attention to the variations where the robber does not have the push ability, but the cops either have the strong push ability or the weak push ability. We would also like to note here that if neither cops nor the robber has the push ability, then this game is equivalent to the classical \CR game on oriented graphs.

In this work we consider oriented graphs whose underlying graphs are  finite and connected. In the classical \CR, we assume connectedness because the cop number of a graph is the sum of the cop numbers of its connected components. Moreover, notice that a disconnected graph cannot be cop-win even if the cop has the strong push ability. Recall that oriented graphs do not admit $2$-cycles. If we allow parallel arcs with opposite orientation, then even for simple graphs such as $4$-cycle with arcs both ways between consecutive vertices, one cop with push ability does not have a winning strategy.

Let $\csp{G}$ be the cop number of $\overrightarrow{G}$ when the cops have the strong push ability and let $\mathsf{c_{wp}}(\overrightarrow{G})$ be the cop number of $\overrightarrow{G}$ when the cops have the weak push ability. It is easy to see that $\mathsf{c_{sp}}(\overrightarrow{G}) \leq \mathsf{c_{wp}}(\overrightarrow{G})$. Das et al.~\cite{das2023cops} also observed that if $\overrightarrow{G}$ is an orientation of a complete multipartite graph,  then $\csp{G} = 1$. 

\smallskip
\noindent\textbf{Push operation.} Pushing vertices of an oriented graph to obtain an orientation that admits some desirable property is well-studied~\cite{klostermeyer1}. 
In particular, it is NP-complete to decide if, using the push operation, an oriented graph can be made a directed acyclic graph (DAG), a strongly connected digraph, or a semi-connected digraph~\cite{klostermeyer1}, and  if an oriented graph can be made Hamiltonian~\cite{klostermeyer2}.  Heard and Huang~\cite{heard2009kernel} established that it is NP-complete to decide if an oriented graph can be made to contain no directed cycle of odd length using the push operation. Moreover, it is NP-complete to decide if an undirected graph is an underlying graph of an oriented graph whose vertices can be pushed to obtain an oriented clique\footnote{an \textit{oriented clique} is an oriented graph where each pair of non-adjacent vertices is connected by a directed 2-path.}~\cite{bensmail2017oriented}. Klostermeyer and {\v{S}}olt{\'e}s~\cite{klostermeyer2000extremal} established that sufficiently large tournaments have high connectivity and can be pushed to obtain an exponential number of Hamiltonian cycles. Pushing operation has been extensively studied in the theory of coloring digraphs and homomorphism of digraphs~\cite{sen2017homomorphisms,bensmail2021pushable,bensmail2023pushable}.

Special attention has been provided to the problem of deciding if  an input oriented graph can be made a DAG using the push operation. Huang, MacGillivray, and Wood~\cite{huang2001pushing} characterized, in terms of forbidden subdigraphs, the multipartite tournaments which can be made acyclic using the push operation. Huang, MacGillivray, and Yeo~\cite{huang2002pushing} proved that the problem remains NP-complete even when restricted to bipartite digraphs and characterized, in terms of two forbidden subdigraphs, the chordal digraphs and bipartite permutation digraphs which can be made acyclic using the push operation. For the forbidden subdigraph characterization for bipartite permutation graphs, Rizzi~\cite{rizzi2006acyclically} provided a polynomial time algorithm. 

Observe that if a digraph $\overrightarrow{G}$ can be made a DAG with a single source vertex  using the push operation, then $\csp{G} = 1$ as the cop can start at the vertex designated to be the source vertex and push the graph to be a DAG with a single source over the next few rounds. Then, the cop can capture $\R$ in this DAG. We begin with establishing that this observation extends to every oriented graph that can be made a DAG using the push operation. To prove this result, we establish that if we can push a digraph to be a DAG, then we can push it to be a DAG with a single source (in \cref{T:DAG}). In particular, we have the following theorem concerning DAGs. 
\begin{restatable}{theorem}{PushDag}\label{C:DAG}
    If $\overrightarrow{G}$ can be made a DAG using the push operation, then $\csp{G} = 1$.
\end{restatable}

Thus, the class of directed graphs where a cop with strong push ability can win is a superclass of graphs that can be made DAGs using the push operation. We are interested in this class of graphs. Das et al.~\cite{das2023cops} established that orientations of subcubic and of interval graphs are in this class of graphs, i.e., if $\overrightarrow{G}$ is an orientation of a subcubic or an interval graph, then $\csp{G} = 1$. To establish their result concerning subcubic graphs, Das et al.~\cite{das2023cops} proved the following, which shall be useful to us as well.

\begin{proposition}[\cite{das2023cops}]\label{P:degenerate}
   Let $v$ be a vertex of an oriented graph $\overrightarrow{G}$ such that $|N^+(v)|+|N^-(v)| \leq 3$. Moreover, let $\overrightarrow{H}$ be the induced subgraph of $\overrightarrow{G}$ obtained by deleting $v$. If $\csp{H} =1$, then $\csp{G}=1$.
\end{proposition}

The above result argues that removal of vertices with degree at most $3$ does not change the outcome of the game considered similarly to how deletion of \textit{corners}\footnote{A vertex $v$ is a corner if there is a vertex $u$ such that $N[v]\subseteq N[u]$.} does not change the outcome of the game in undirected graphs.  Observe that \cref{P:degenerate} can be used directly to extend the result for subcubic graphs to 3-degenerate graphs. Although it is not mentioned explicitly in~\cite{das2023cops}, the result should be considered a corollary of their result. We still provide a proof for the sake of completeness.
\begin{corollary}\label{T:3-degenrate}
    Let $\overrightarrow{G}$ be an oriented graph such that its underlying graph $G$ is a 3-degenerate graph. Then, $\csp{G} = 1$.
\end{corollary}
\begin{proof}
Consider a minimal graph $ \overrightarrow{G}$ such that its underlying graph $G$ is 3-degenerate and $\csp{G}>1$ (i.e., for every proper induced subgraph $\overrightarrow{H}$ of $\overrightarrow{G}$,  $\csp{H} = 1$). Note that $\overrightarrow{G}$ contains at least two vertices as a single vertex graph is trivially cop-win.  Since $G$ is a 3-degenerate graph, there is at least one vertex $v\in V(\overrightarrow{G})$ such that $|N^+(v)|+|N^-(v)| \leq 3$. Let $\overrightarrow{H}$ be the induced subgraph of $\overrightarrow{G}$ we get after deleting the vertex $v$. Since $\overrightarrow{G}$ is a minimal graph with $\csp{G}>1$, we have that $\csp{H} = 1$. Then, due to \cref{P:degenerate}, $\csp{G} =1$, which contradicts our assumption that $\csp{G}>1$. 
Therefore, if $\overrightarrow{G}$ is an orientation  of a 3-degenerate  graph, then $\csp{G} =1$.
\end{proof} 

One can obtain using Euler's formula that triangle-free planar graphs are 3-degenerate. Moreover, outerplanar graphs are 2-degenerate~\cite{outDegenerate} and 3-dimensional grids are 3-degenerate. Thus, \cref{T:3-degenrate} immediately adds these graph classes to the list of graph classes whose orientations are cop-win if the cop has the strong push ability.



We further extend the result concerning subcubic graphs to establish that if $\overrightarrow{G}$ is an orientation of a graph $G$ with maximum degree $4$, then $\csp{G} = 1$ in \cref{S:maxdeg}.  We remark that subcubic graphs (and graphs of bounded degree) are well-studied in the case of undirected graphs as well and they are shown to have unbounded cop number~\cite{hosseini2021meyniel} and even are shown to be Meyniel-extremal\footnote{Meyniel's conjecture~\cite{frankl1987cops} asserts that the cop number of any undirected graph is $\mathcal{O}(\sqrt{n})$. A graph class is said to be Meyniel-extremal if it contains some graphs with cop number $\Omega(\sqrt{n})$.}. 
To prove our result, we establish (in \cref{S:maxdeg}) that if $\overrightarrow{G}$ is an orientation of a $4$-regular graph, then $\csp{G} =1$. In particular, we have the following lemma.

\begin{restatable}{lemma}{Tregular}\label{T:4-regular}
    If $\overrightarrow{G}$ is an orientation of a $4$-regular graph, then $\csp{G} = 1$.
\end{restatable}

Then, using \cref{T:4-regular}, \cref{P:degenerate}, and \cref{T:3-degenrate}, we obtain the following result.

\begin{restatable}{theorem}{Tmaxdegree}\label{T:maxdegree}
    Let $\overrightarrow{G}$ be an oriented graph and $G$ be its underlying graph. If $\Delta(G) \leq 4$, then $\csp{G} = 1$.
\end{restatable}



\medskip
\noindent\textbf{Related Work.}
The \CR game is well studied on both directed and undirected graphs. Hamidoune \cite{hamidoune} considered the game on Cayley digraphs. 
Frieze et al.~\cite{lohjgt} studied the game on digraphs and gave an upper bound of $\mathcal{O}\left(\frac{n(\log \log n)^2}{\log n}\right)$ for cop number in digraphs. 
Loh and Oh~\cite{loh} considered the game on strongly connected planar digraphs and proved that every $n$-vertex strongly connected planar digraph has cop number $\mathcal{O}(\sqrt{n})$. Moreover, they constructively proved the existence of a strongly connected planar digraph with cop number greater than three, which is in contrast to the case of undirected graphs where the cop number of a planar graph is at most three~\cite{fromme1984game}. 
The computational complexity of determining the cop number of a digraph (and undirected graphs also) is a challenging question in itself. Goldstein and Reingold~\cite{exptime} proved that deciding whether $k$ cops can capture a robber is EXPTIME-complete for a variant of \CR and conjectured that the same holds for the classical \CR as well. 
Later, Kinnersley~\cite{kinnersley2015cops} proved that conjecture and established that determining the cop number of a graph or digraph is EXPTIME-complete. 
Kinnersley~\cite{kinnersley2018bounds} also showed that $n$-vertex strongly connected cop-win digraphs can have capture time $\Omega(n^2)$, whereas for undirected cop-win graphs the capture time is at most $n-4$ moves~\cite{gavenvciak2010cop}. 

Hahn and MacGillivray~\cite{hahn} gave an algorithmic characterization of the cop-win finite reflexive digraphs and showed that any $k$-cop game can be reduced to $1$-cop game, resulting in an algorithmic characterization for $k$-copwin finite reflexive digraphs. 
However, these results do not give a structural characterization of such graphs. 
Darlington et al.~\cite{darlington2016cops} tried to structurally characterize cop-win oriented graphs and gave a conjecture that was later disproved by Khatri et al.~\cite{khatri2018study}, who also studied the game in oriented outerplanar graphs and line digraphs. Moreover, several variants of the game on directed graphs depending on whether the cop/robber player has the ability to move only along or both along and against the orientations of the arcs are also studied~\cite{das2021cops}. Gahlawat, Myint, and Sen~\cite{gahlawat2023cops} considered these variants under the classical operations like subdivisions and retractions, and also established that all these games are NP-hard even for $2$-degenerate bipartite graphs with arbitrary high girth.

Recently, the cop number of planar Eulerian digraphs and related families was studied in several articles \cite{mohar2,   mohar}. In particular, Hosseini and Mohar~\cite{mohar} considered the orientations of integer grid that are vertex-transitive and showed that at most four cops can capture the robber on arbitrary finite quotients of these directed grids. De la Maza et al.~\cite{mohar2} considered the \textit{straight-ahead} orientations of 4-regular quadrangulations
of the torus and the Klein bottle and proved that their cop number is bounded
by a constant. They also showed that the cop number of every
$k$-regularly oriented
toroidal grid is at most 13.

Bradshaw et al.~\cite{bradshaw2021cops} proved that the cop number of directed and undirected Cayley graphs on abelian groups is $\mathcal{O}(\sqrt{n})$, and constructed families of graphs and digraphs which are Cayley graphs of abelian groups with cop number $\Theta(\sqrt{n})$ for every $n$. The family of digraphs thus obtained has the largest cop number in terms of $n$ of any known digraph construction.

\section{Preliminaries}\label{S:prelim}
In this paper, we consider the game on oriented graphs whose underlying graph is simple, finite, and connected. Let $\overrightarrow{G}$ be an oriented graph with $G$ as the underlying undirected graph of $\overrightarrow{G}$. We say that $\overrightarrow{G}$ is an \textit{orientation} of $G$. We consider the push operation on the vertices of $\overrightarrow{G}$, and hence the orientations of arcs in $\overrightarrow{G}$ might change. We generally use $\overrightarrow{G}$ to refer to the graph G with its orientation at some specific point in the game. Note that although the orientations of the arcs in $\overrightarrow{G}$ might change, the underlying graph $G$ remains the same. Moreover, it is worth noting that given $\overrightarrow{G}$ and $\overrightarrow{H}$ such that $\overrightarrow{G}$ and $\overrightarrow{H}$ have the same underlying graph, it might be possible that there is no sequence of pushing vertices in $\overrightarrow{G}$ that yields $\overrightarrow{H}$. 

Given an undirected graph $G$ and $v\in V(G)$, let $N(v) = \set{u~|~uv\in E(G)}$ and $N[v] = N(v)\cup \set{v}$. Observe that if $\overrightarrow{G}$ is an orientation of $G$, then $N(v) = N^+(v) \cup N^-(v)$. The \textit{degree} of a vertex $v$, denoted $d(v)$ is $|N(v)|$. The \textit{maximum degree} of $G$, denoted $\Delta(G)$,  is $\max_{v\in V(G)} d(v)$.  A graph $G$ is \textit{subcubic} if $\Delta(G) \leq 3$. Let $k\in \mathbb{N}$. A graph $G$ is said to be $k$-\textit{regular} if for each $v\in V(G)$, $d(v) = k$. A graph $G$ is $k$-\textit{degenerate} if any induced subgraph of $G$ contains at least one vertex of degree at most $k$.  

Let $v$ be a vertex of $\overrightarrow{G}$ and $S$ be a subset of vertices of $\overrightarrow{G}$ (i.e., $S\subseteq V(\overrightarrow{G})$). Then, we say that $v$ is a {\em source in $S$} if $S\subseteq N^+[v]$. Moreover, we say that $|N^+(v)|$ is the \textit{out-degree} of $v$, $|N^-(v)|$ is the \textit{in-degree} of $v$, and $|N^+(v)|+|N^-(v)|$ is the \textit{degree} of $v$. A vertex $w$ is said to be \textit{reachable from} $u$ if either $w=u$ or there is a directed path from $u$ to $w$. 

When we have a single cop, we denote the cop by $\C$. We denote the robber by $\R$ throughout the paper. 
If $\R$ is at a vertex $v$ such that $N^+(v) = 0$, then we say that $\R$ is \textit{trapped} at $v$. The following result shall be useful to us.
\begin{proposition}[\cite{das2023cops}]\label{P:trap}
    Let $\overrightarrow{G}$ be an oriented graph and $\R$ is trapped at a vertex $v\in V(\overrightarrow{G})$. Then, a cop with (strong or weak) push ability can capture $\R$ in a finite number of rounds.
\end{proposition}

\section{Graphs Pushable to be DAGs}
In this section, we establish that if $\overrightarrow{G}$ can be made a DAG using the push operations, then $\csp{G} = 1$. To this end, first we prove the following lemma. 
\begin{lemma}\label{L:pushDAG}
    Let $\overrightarrow{G}$ be a DAG and $u$ be a source vertex of $\overrightarrow{G}$. Moreover, let $X$ and $Y\neq \emptyset$ be the sets of vertices that are reachable from $u$ and  not reachable from $u$, respectively. Then, we can push some vertices of $Y$ such that:
    \begin{enumerate}
        \item $\overrightarrow{G}$ is still a DAG and $u$ is still a source vertex in $\overrightarrow{G}$,
        \item each vertex in $X$ is still reachable from $u$, 
        \item and at least one vertex in $Y$ is reachable from $u$.
    \end{enumerate}
\end{lemma}

\begin{proof}
    For any $x\in X$ and $y \in Y$, clearly $\overrightarrow{xy}\notin E(\overrightarrow{G})$ as otherwise $y$ is also reachable from $u$ due to transitivity of reachability. Let $Y'$ be the set of vertices in $Y$ that has at least one out-neighbor in $X$. Observe that $Y'\neq \emptyset$ since the underlying graph of $\overrightarrow{G}$ is connected. Now, push each vertex in $Y$. Notice that this operation reverses the orientation of each arc with one endpoint in $X$ and one endpoint in $Y'$, ensuring that each vertex in $Y'$ is now reachable from $u$. Every other arc has the same orientation still. Thus, each vertex in $X$ is reachable from $u$ and each vertex in $Y'$ is reachable from $u$ (and $|Y'|>0$). Next, observe that $u$ is still a source vertex since no arc incident on $u$ changed its orientation. Finally, observe that $\overrightarrow{G}$ is still a DAG after pushing vertices of $Y$. Indeed, there can be no directed cycle that contains at least one vertex from $X$ and at least one vertex from $Y$ as all arcs between vertices of $X$ and $Y$ are oriented from $X$ to $Y$, and if there is a cycle in $X$ or in $Y$, it would imply that $\overrightarrow{G}$ contained a cycle before pushing $Y$, contradicting the fact that $\overrightarrow{G}$ was a DAG. This completes our proof. 
\end{proof}

Using \cref{L:pushDAG}, we establish that if $\overrightarrow{G}$ is a DAG, then $\overrightarrow{G}$ can be pushed to be a DAG with a single source in the following lemma. 
\begin{restatable}{lemma}{TDAG}\label{T:DAG}
    If $\overrightarrow{G}$ is a DAG, then $\overrightarrow{G}$ can be made a DAG with a single source using the push operations. 
\end{restatable}
\begin{proof}
    Since $\overrightarrow{G}$ is a DAG, it contains at least one source vertex, say $u$. If every vertex in $\overrightarrow{G}$ is reachable from $u$, then we have nothing to prove as $\overrightarrow{G}$ is a DAG with a single source. Otherwise, let $Z$ be the set of vertices that are not reachable from $u$. Then, we apply \cref{L:pushDAG} iteratively to ensure that after each application of \cref{L:pushDAG}, at least one vertex of $Z$ that was not reachable from $u$ before this application of \cref{L:pushDAG} is now reachable from $u$ while ensuring that (i) $\overrightarrow{G}$ remains a DAG, (ii) $u$ remains a source vertex, and (iii) every vertex that was reachable in a previous round is still reachable. Since $Z\subseteq V(\overrightarrow{G})$  is finite and at least one vertex from $Z$ that was not reachable before becomes reachable in each application of \cref{L:pushDAG}, after finitely many applications of \cref{L:pushDAG}, $\overrightarrow{G}$ will be a DAG with a single source $u$.
\end{proof}

Finally, we have the following theorem.
\PushDag*
\begin{proof}
    First, push the vertices of  $\overrightarrow{G}$ such that now $\overrightarrow{G}$ is oriented to be a DAG. Second, apply \cref{T:DAG} to orient $\overrightarrow{G}$ as a DAG with a single source using the push operations. Now, let $u$ be the unique source of $\overrightarrow{G}$. Since DAGs with a unique source are cop-win when $\R$ has to follow the orientations of the arcs~\cite{das2021cops}, the cop has a winning strategy in $\overrightarrow{G}$ by starting the game at $u$, pushing the vertices of $\overrightarrow{G}$ to make it a DAG with a unique source $u$, and then capturing the robber using its winning strategy for DAGs with a single source.  This completes our proof. 
\end{proof}

\section{Graphs with Maximum Degree 4}\label{S:maxdeg}
In~\cite{das2023cops}, the authors established that if the underlying graph of an oriented graph $\overrightarrow{G}$ is subcubic, then $\csp{G} = 1$. In this section, we extend this result to show that if $\overrightarrow{G}$ is an orientation of a graph $G$ with $\Delta(G) \leq 4$, then $\csp{G} = 1$. Since \cref{T:4-regular}, along with \cref{P:degenerate} and \cref{T:3-degenrate}, imply \cref{T:maxdegree}, most of this section is devoted to establish the result for the orientations of $4$-regular graphs (\cref{T:4-regular}).  

A vertex $v \in V(\overrightarrow{G})$ is said to be \textit{visited} if $v$ was occupied by $\R$ in some previous round of the game. The current position of the robber is also considered visited.    Roughly speaking, the cop will apply the following strategy to capture $\R$:  the cop ensures that every vertex that was previously visited by the robber is ``no longer habitable'' for the robber in the sense that if $\R$ moves to a vertex that was previously visited, then $\R$ will be trapped after at most two more moves of $\R$. Thus, the movements of $\R$ in the graph are almost ``monotone'' except for its last few moves. To ensure this, $\C$ will maintain the following invariant: for each visited vertex $u$, $|N^+(u)| \leq 1$. If this invariant is broken, then $\R$ will be trapped after finitely many moves of $\R$. First, we have the following easy lemma. 
\begin{lemma}\label{L:trivial}
    If $\R$ occupies a vertex $v\in V(\overrightarrow{G})$ such that $|N^+(v)| \leq 1$ on a cop's move, then $\R$ will be trapped in the next cop move.
\end{lemma}
\begin{proof}
    If $|N^+(v)| = 0$, then this statement is trivially true. Otherwise, let $N^+(v) = \set{u}$. Thus, $\C$ can trap $\R$ by pushing $u$.
\end{proof}

Next, we have the following remark.
\begin{remark}\label{R:trivial}
     If $\R$ occupies a vertex $v$ such that $|N^+(v)| = 1$ on the robber's turn, then $\R$ is forced to move to the unique out-neighbor of $v$, since otherwise, the cop can trap $\R$ by \cref{L:trivial}. Whenever $\R$ occupies a vertex $v$ such that $|N^+(v)|=2$ and $\C$ pushes one vertex of $N^+(v)$, we will assume that $\R$ will move to the other out-neighbor of $v$ to avoid getting trapped, and we will say that $\R$ is {\em forced} to move to the other out-neighbor  of $v$.
\end{remark}

In the following lemma, we establish that $\C$ can maintain the invariant that for each visited vertex $v$, $|N^+(v)|\leq 1$ and whenever this invariant breaks, $\C$ will trap $\R$ in a finite number of rounds, and hence catch $\R$ after finitely many rounds (due to \cref{P:trap}).

\begin{lemma}\label{L:regular}
    Let $\overrightarrow{G}$ be an orientation of a $4$-regular graph $G$. Let $\R$ occupy a vertex $u\in V(\overrightarrow{G})$ on a cop's move. Then, after the cop's move, either, 
    \begin{enumerate}
        \item the following invariant is satisfied: for each visited vertex $w$, $|N^+(w)| \leq 1$, 
        \item or $\C$ will trap $\R$ in a finite number of rounds.
    \end{enumerate}
\end{lemma}
\begin{proof}
    We will prove this result using induction on the number of moves played. For the base case, we show that if $\R$ begins at the vertex $u$, then the statement of the lemma holds. Let $\R$ begin the game at $u$. At this point, if $|N^+(u)| = 4$, then $\C$ can trap $\R$ by pushing $u$. If $|N^+(u)| = 3$, then $\C$ can achieve (1) by pushing $u$. If $|N^+(u)| = 2$, then let $N^+(u)= \set{x,y}$. Here, $\C$ will push the vertex $x$, forcing $\R$ to move to $y$. Notice that  (1) is achieved in this case. If $|N^+(u)|\leq 1$, then $\C$ will trap $\R$ using Lemma~\ref{L:trivial}. 

    Now, suppose the condition (1) of our lemma is satisfied when $\R$ is at a vertex $v$ (i.e., each visited vertex, including $v$ has out-degree at most one after the cop's move). Now, let $\R$ move from $v$ to some vertex $u\in N^+(v)$. Thus, $N^+(v) = \set{u}$. Now, we distinguish the action of $\C$ based on the cardinality of $|N^+(u)|$. If $|N^+(u)| \leq 1$, then $\C$ can trap $\R$ using Lemma~\ref{L:trivial}. If $|N^+(u)| = 3$, then $\C$ will push $u$. Note that, now, $N^+(u) = \set{v}$ and $N^+(v) = \emptyset$ (since pushing $u$ can change the orientation of only one arc incident to $v$, i.e.,  $\overrightarrow{vu}$ is changed to $\overrightarrow{uv}$). Thus, if $\R$ moves to $v$, then $\R$ is automatically trapped at $v$. Otherwise, $\C$ will push $v$ in the next round to trap $\R$ at $u$. 
    
    Hence, from now on, we can assume that $|N^+(u)| = 2$ (when $\R$ moves to $u$) and let $N^+(u) = \set{x,y}$. Now, we will distinguish two cases  depending on whether $xy\in E(G)$ or not. First, we prove the following claim.

    \begin{claim}\label{C:edge}
        If $xy\in E(G)$, then $\C$ can trap $\R$ after a finite number of rounds.
    \end{claim}
    \begin{proofofclaim}
        Without loss of generality, let us assume that $\overrightarrow{xy} \in E(\overrightarrow{G})$ when $\R$ moved from $v$ to $u$. Now, we have the following three exhaustive cases. 

        \begin{figure}
\centering
\begin{subfigure}{.3\textwidth}
  \centering
  \includegraphics[width=.8\linewidth] {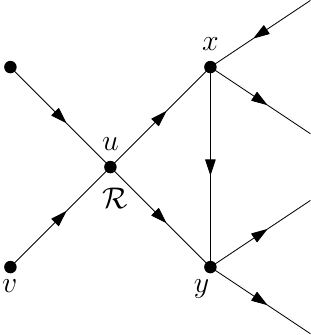}
  \caption{$\R$ moves to $u$.}
  \label{fig:OC1}
\end{subfigure}\hfill
\begin{subfigure}{.3\textwidth}
  \centering
  \includegraphics[width=.8\linewidth] {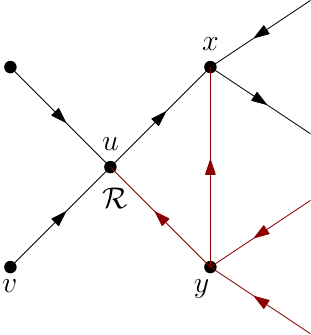}
  \caption{$\C$ pushes $y$.}
  \label{fig:OC2}
\end{subfigure}\hfill
\begin{subfigure}{.3\textwidth}
  \centering
  \includegraphics[width=.8\linewidth] {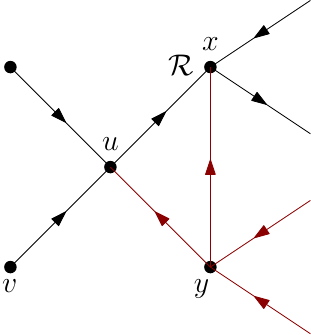}
  \caption{$\R$ moves to $x$.}
  \label{fig:OC3}
\end{subfigure}\hfill

\caption{An illustration for Case~1 of \cref{C:edge}. Once $\R$ moves to $x$,  $\C$ traps $\R$ using \cref{L:trivial}} 
\label{fig:Ecase1}
\end{figure}
        
        \medskip
        \noindent\textbf{Case~1.} $x$ has at most one out-neighbor other than $y$ (i.e., $|N^+(x)| \leq 2$): See Figure~\ref{fig:Ecase1} for an illustration. In this case, $\C$ begins with pushing $y$ when $\C$ moves to $u$. This forces $\R$ to move to $x$. Now, $|N^+(x)| \leq 1$ (since the orientation of $\overrightarrow{xy}$ was reversed to $\overrightarrow{yx}$ when $\C$ pushed $y$). Hence, $\C$ can trap $\R$ using \cref{L:trivial}.

        \medskip
        \noindent\textbf{Case~2.} $y$ has at most one out-neighbor (i.e., $|N^+(y)| \leq 1$): Due to Case~1, we can safely assume that $|N^+(x)| = 3$, else $\C$ can trap $\R$ using the strategy from Case~1. See \cref{fig:Ecase2} for an illustration. $\C$ begins with pushing the vertex $x$, and now $N^+(x) =\set{u}$ and $|N^-(x)|=3$. Thus, $\R$ is forced to move to $y$. At this point, either $|N^+(y)| = 1$ (if $|N^+(y)|=0$ when $\C$ pushed $x$) or  $|N^+(y)| = 2$ and $x\in N^+(y)$ (if $|N^+(y)|=1$ when $\C$ pushed $x$). If $|N^+(y)| = 1$, then $\C$ will trap $\R$ using \cref{L:trivial}.
        
        If $|N^+(y)| =2$, i.e., $N^+(y) = \{w,x\}$, when $\R$ moved to $y$, then $\C$ will use the following strategy to trap $\R$. Now, $\C$ pushes $w$, forcing $\R$ to move to $x$. At this point, consider the neighborhood of $x$. Initially, $|N^+(x)| = 3$ and $N^-(x) =\set{u}$. Then, we pushed $x$ and we had $|N^-(x)| = 3$ and $N^+(x) = \{u\}$. After this, we pushed the vertex $w$. Now, if $xw$ is not an edge in the underlying graph (i.e., $\overrightarrow{xw}$ was not an arc before we pushed $x$), then we have that $N^+(x) = \set{u}$ when $\R$ moved to $x$, and hence, $\C$ can trap $\R$  using \cref{L:trivial}. Otherwise, if $xw$ is an edge in the underlying graph (i.e., $\overrightarrow{xw}$ and $\overrightarrow{yw}$ were arcs before we pushed $x$), then we have that $N^+(x) = \set{u,w}$ when $\R$ moves to $x$. At this point, we again push $w$, forcing $\R$ to move to $u$. Now, consider the neighborhood of $u$ and sequence of push operations performed. Initially (in the beginning of our case), $N^+(u) = \set{x,y}$. Then, $\C$ pushed $x$ once and $w$ twice. Hence, the only effect in the graph is that of pushing $x$, and $N^+(u) = \set{y}$ when $\R$ moves to $u$. Hence, $\R$ will be trapped by an application of \cref{L:trivial}.  

        \begin{figure}
\centering
\begin{subfigure}{.3\textwidth}
  \centering
  \includegraphics[width=.8\linewidth] {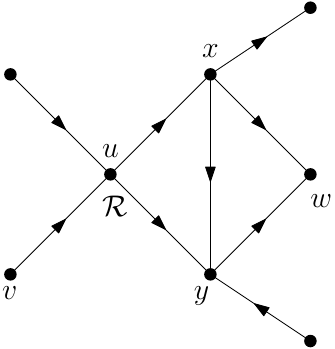}
  \caption{$\R$ moves to $u$.}
  \label{fig:OC11}
\end{subfigure}\hfill
\begin{subfigure}{.3\textwidth}
  \centering
  \includegraphics[width=.8\linewidth] {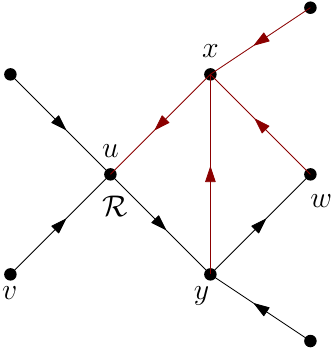}
  \caption{$\C$ pushes $x$.}
  \label{fig:OC22}
\end{subfigure}\hfill
\begin{subfigure}{.3\textwidth}
  \centering
  \includegraphics[width=.8\linewidth] {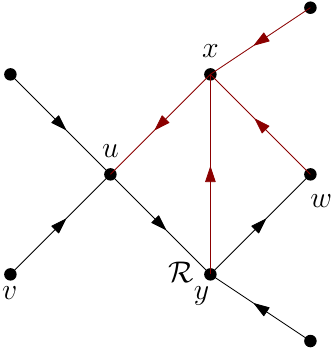}
  \caption{$\R$ moves to $y$.}
  \label{fig:OC33}
\end{subfigure}\hfill
\begin{subfigure}{.3\textwidth}
  \centering
  \includegraphics[width=.8\linewidth] {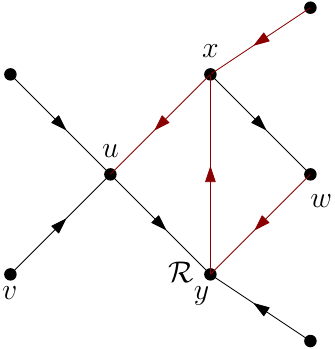}
  \caption{$\C$ pushes $w$.}
  \label{fig:OC44}
\end{subfigure}\hfill
\begin{subfigure}{.3\textwidth}
  \centering
  \includegraphics[width=.8\linewidth] {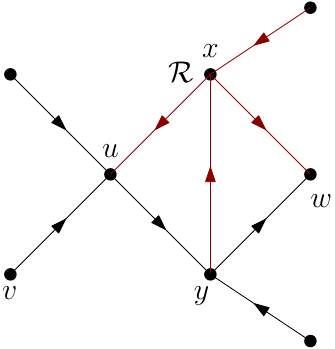}
  \caption{$\R$ moves to $x$ and $\C$ pushes $w$.}
  \label{fig:OC35}
\end{subfigure}\hfill
\begin{subfigure}{.3\textwidth}
  \centering
  \includegraphics[width=.8\linewidth] {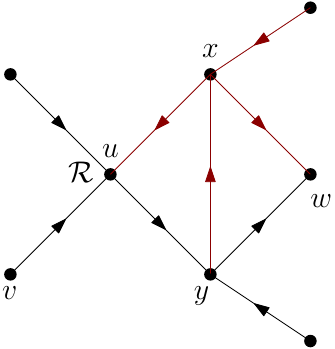}
  \caption{$\R$ moves to $x$.}
  \label{fig:OC36}
\end{subfigure}\hfill
\caption{An illustration for Case~2 of \cref{C:edge}. We illustrate the case when $w\in N^+(x) \cap N^+(y)$. \cref{fig:OC35} illustrates two steps: $\R$ moving from $y$ to $x$ and then $\C$ pushing $w$. Once $\R$ moves to $u$ from $x$,  $\C$ traps $\R$ using \cref{L:trivial}} 
\label{fig:Ecase2}
\end{figure}

        \medskip
        \noindent\textbf{Case~3.} $|N^+(x)| = 3$ and $|N^+(y)| = 2$: In this case, $\C$ begins with pushing $x$, forcing   $\R$  to move to $y$. Notice that when $\R$ moves to $y$, $|N^+(y)| = 3$ (since the orientation of $\overrightarrow{xy}$ was reversed when we pushed $x$), $N^-(y) = \{u\}$, and $N^+(u) = \{y\}$. Now, $\C$ pushes $y$ to make $N^+(y) = \set{u}$ and $N^+(u) = \emptyset$, forcing   $\R$ to move to $u$, where $\R$ is trapped  since $N^+(u) = \emptyset$. 
    
    \medskip  
    Finally, to complete the proof of our claim, we observe that our 3 cases are indeed exhaustive when $\overrightarrow{xy}$ is an arc. Since our underlying graph $G$ is 4-regular and $\overrightarrow{ux}, \overrightarrow{uy}, \overrightarrow{xy}$ are arcs, clearly we have that $|N^+(x)| \leq 3$ and $|N^+(y)| \leq 2$. We consider the case $|N^+(x)| = 3$ and $|N^+(y)| = 2$ in Case~3, $|N^+(x)|\leq 2$ in Case~2, and $|N^+(x)|\leq 1$ in Case~1. This completes the proof of our claim.
    \end{proofofclaim}

    Hence, due to \cref{C:edge}, we can assume that $xy$ is not an edge in the underlying graph $G$. Next, we prove an easy but useful claim. 

    \begin{claim}\label{C:neighborvisited}
        Given $xy\notin E(G)$, if $|N^+(x)| \neq 2$ or $|N^+(y)| \neq 2$, then $\C$ can trap $\R$ in at most $2$ robber moves.
    \end{claim}
    \begin{proofofclaim}
        Without loss of generality, let us assume that $|N^+(x)|\neq 2$. $\C$ begins with pushing $y$, forcing $\R$ to move to $x$. Notice that this does not change $|N^+(x)|$ since $xy\notin E(G)$. We distinguish the following two cases:

        \medskip
        \noindent\textbf{Case~1.} $|N^+(x)|\leq 1$: $\C$ can trap $\R$ using \cref{L:trivial}.

        \medskip
        \noindent\textbf{Case~2.} $|N^+(x)| =3$. $\C$ pushes $x$, and now $N^+(x) = \set{u}$. This forces $\R$ to move to $u$. Since we have pushed both out-neighbors of $u$, $N^+(u) = \emptyset$, and thus $\R$ is trapped. 
    \end{proofofclaim}

    In the following claim, we establish that if $xy \notin E(G)$, then $\C$ can establish the conditions of our lemma. Observe that the following claim, along with \cref{C:edge} will complete our proof.
    \begin{claim}\label{C:nonEdge}
        If $xy\notin E(G)$, then $\C$ can establish the conditions of our lemma.
    \end{claim}
    \begin{proofofclaim}
        Due to \cref{C:neighborvisited}, unless $|N^+(x)| =|N^+(y)| = 2$, $\C$ can trap $\R$ in at most two moves of the robber. Hence, we assume that $|N^+(x)| =|N^+(y)| = 2$ (and hence neither $x$ nor $y$ is a visited vertex) for the rest of this proof. Further, if no vertex of $N^+(x)$ (resp. $N^+(y)$) is a visited vertex, then $\C$ pushes $x$ (resp. $y$), forcing $\R$ to move to $y$ (resp. $x$), and this satisfies condition~(1) of our lemma. Therefore, for the rest of this proof, we assume that $|N^+(x)| =|N^+(y)| = 2$ and that each of $N^+(x)$ and $N^+(y)$ contains a visited vertex. Further, let $N^+(x) = \set{x_1,x_2}$ and $N^+(y)= \set{y_1,y_2}$, and without loss of generality, let us assume that $x_1$ and $y_1$ are visited vertices.  We remark that possibly $\set{x_1,x_2} \cap \set{y_1,y_2} \neq \emptyset$, and we assume that whenever $x_1\in \{y_1,y_2\}$, we set $x_1=y_1$.

        Now, we distinguish the following two cases depending on whether at least one of $x_1x_2$ or $y_1y_2$ is an edge in the underlying graph:

        \medskip
        \noindent\textbf{Case~1.} At least one of $x_1x_2$ or $y_1y_2$ is an edge in the underlying graph $G$ of $\overrightarrow{G}$: Without loss of generality, let us assume that $x_1x_2 \in E(G)$. Now, again we have the following two cases depending on the orientation of $x_1x_2$ when $\R$ moves to $u$ from $v$. 
        \begin{enumerate}
            \item Edge $x_1x_2$ is oriented as $\overrightarrow{x_1x_2}$ when $\R$ moves to $u$: 
            Observe that, in this case, $N^+(x_1) = \set{x_2}$ when $\R$ moved to $u$ (since $x_1$ is a visited vertex and hence, $|N^+(x_1)| \leq 1$). $\C$ begins with pushing the vertex $y$, forcing $\R$ to move to $x$. Observe that, at this point, $|N^+(x_1)| \leq 2$ as possibly $\overrightarrow{yx_1}$ was an arc when $\R$ moved from $v$ to $u$ and then we pushed $y$.  Next, $\C$ pushes $x_2$, which ensures $|N^+(x_1)| \leq 1$ again. Thus,  $\R$ is forced to move to $x_1$, where $\R$ will be trapped via an application of \cref{L:trivial} as $|N^+(x_1)| \leq 1$. 

            \item Edge $x_1x_2$ is oriented as $\overrightarrow{x_2x_1}$ when $\R$ moves to $u$: We again distinguish the following cases depending on how $\set{x_1,x_2}$ and $\set{y_1,y_2}$ intersect. 
            \begin{enumerate}
                \item $x_1=y_1$ and $x_2=y_2$:   Since $x_1=y_1$ is a visited vertex $|N^+(x_1)| \leq 1$ when $\R$ moves to $u$ to from $v$ as per induction hypothesis. Now, we will again consider two cases depending on whether $|N^+(x_1)|= 1$ or $|N^+(x_1)| = 0$.

\begin{figure}
\centering
\begin{subfigure}{.47\textwidth}
  \centering
  \includegraphics[width=.9\linewidth] {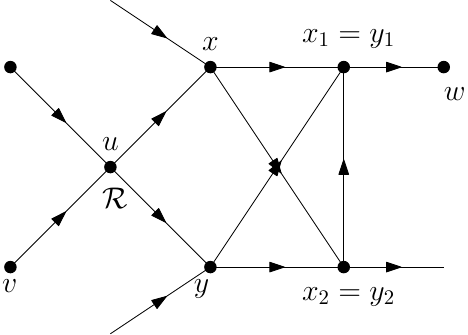}
  \caption{$\R$ moves to $u$.}
  \label{fig:S1}
\end{subfigure}\hfill
\begin{subfigure}{.47\textwidth}
  \centering
  \includegraphics[width=.8\linewidth] {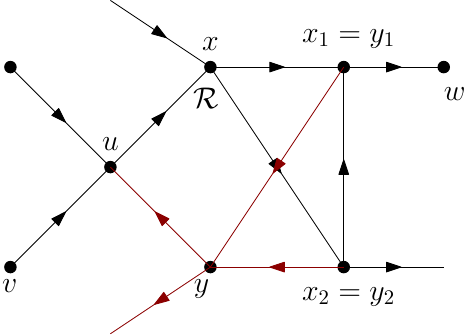}
  \caption{$\C$ pushes $y$ and $\R$ moves to $x$.}
  \label{fig:S2}
\end{subfigure}\hfill
\begin{subfigure}{.47\textwidth}
  \centering
  \includegraphics[width=.8\linewidth] {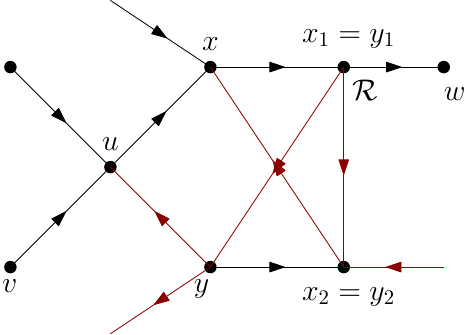}
  \caption{$\C$ pushes $x_2$ and $\R$ moves to $x_1$.}
  \label{fig:S3}
\end{subfigure}\hfill
\begin{subfigure}{.47\textwidth}
  \centering
  \includegraphics[width=.8\linewidth] {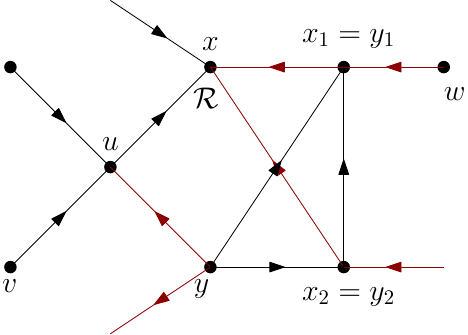}
  \caption{$\C$ pushes $x_1$ and $\R$ moves to $x$.}
  \label{fig:S4}
\end{subfigure}\hfill
\caption{An illustrative proof for the Case~1(2)(a) of \cref{C:nonEdge} when $|N^+(x_1)|=1$. When $\R$ moves to $x$ in \cref{fig:S4}, it gets trapped.} 
\label{fig:sameNeighbor}
\end{figure}

                First, suppose $|N^+(x_1)| = 1$ and let $N^+(x_1) = \set{w}$. See Figure~\ref{fig:sameNeighbor} for an illustration of the proof. $\C$ begins with pushing $y$, forcing $\R$ to move to $x$. Next, $\C$ pushes $y_2$, forcing to move to $x_1$. Now, at this point observe that $N^+(x_1) = \set{w,x_2,y}$ and $N^-(x_1) =\set{x} $. Now, $\C$ pushes $x_1$ to make $N^+(x_1) = \set{x}$, forcing $\R$ to  move to $x$. At this point, observe that $N^+(x) = \emptyset$ as both out-neighbors of $x$ (i.e., $x_1$ and $x_2$) are pushed exactly once and no in-neighbor of $x$ is pushed. Thus, $\R$ gets trapped at $x$. 

                Second, suppose $|N^+(x_1)| = 0$. Here, $\C$ begins with pushing $y$, which forces $\R$ to move to $x$. Next, $\C$ pushes $y$ again, returning the graph to its initial configuration. Now, $\R$ can either (i) stay on $x$, (ii) move to $x_1$, or (iii) move to  $x_2$. (i) If $\R$ stays on $x$, then $\C$ pushes $x_2$, forcing $\R$ to move to $x_1$, and since $|N^+(x_1)| \leq 1$ at this point of time, it will be trapped at $x_1$ due to \cref{L:trivial}. (ii) If $\R$ moves to $x_1$, then observe that it is trapped since $|N^+(x_1)|=0$. (iii) If $\R$ moves to $x_2$, then either $|N^+(x_2)| =1$, i.e., $N^+(x_2) = \set{x_1}$, or $|N^+(x_2)| = 2$, i.e., $N^+(x_2) = \set{x_1,w'}$ for some $w'\in V(\overrightarrow{G})$. If $|N^+(x_2)|=1$, then $\R$ gets trapped via an application of \cref{L:trivial}. If $N^+(x_2) = \set{x_1,w'}$, then cop pushes $w'$, forcing $\R$ to move to $x_1$. Observe that $|N^+(x_1)|\leq 1$ at this point since $|N^+(x_1)| =0$ before $\C$ pushed $w'$. Hence, $\R$ will be trapped via an application of \cref{L:trivial}.

                \item $x_1=y_1$ but $x_2\neq y_2$: Since $x_1$ is a visited vertex, similarly to case above, we have that either $|N^+(x_1)| = 0$ or $|N^+(x_1)| = 1$ and we distinguish cop's strategy based on this.

                First, let $N^+(x_1) = \set {w}$ when $\R$ moved to $u$. $\C$ begins with pushing the vertex $y$, which forces $\R$ to move to $x$. Now, observe that $N^+(x_1) = \set{w,y}$. When $\R$ reaches $x$, $\C$ pushes $x_2$, forcing $\C$ to move to $x_1$. At this point, observe that $N^+(x_1) = \set{x_2,y,w}$ (since $x_1x_2\in E(G)$) and $N^+(x) =\set{x_1}$. When $\R$ reaches $x_1$, $\C$ pushes $x_1$, which makes  $N^+(x_1) = \set{x}$, and hence $\R$ is forced to moved to $x$. When $\R$ moves to $x$, $|N^+(x)|=0$, and hence $\R$ is trapped at $x$.

                Second, let $N^+(x_1) = \emptyset$. Again $\C$ begins with pushing $y$, forcing $\R$ to move to $x$. $\C$ pushes $y$ again returning the graph to its original configuration. Now $\R$ has one of the following three options: (i) $\R$ moves to $x_1$, in which case it gets trapped by definition since $|N^+(x_1)| = 0$.
                (ii) $\R$ stays on $x$, in which case $\C$ pushes $x_2$, forcing $\R$ to move to $x_1$ while $|N^+(x_1)| = 1$, and hence $\R$ will be trapped due to \cref{L:trivial}. (iii) $\R$ moves to $x_2$. Now, observe that $|N^+(x_2)| \in \set{1,2,3}$. If $|N^+(x_2)| =1$, then $\R$ is trapped using \cref{L:trivial}. If $|N^+(x_2)|=2$, then $\C$ pushes the out-neighbor of $x_2$ other than $x_1$, forcing $\R$ to move to $x_1$, and at this point, observe that $|N^+(x_1)| \leq 1$, and hence $\R$ will be trapped using \cref{L:trivial}. Finally, if $|N^+(x_2)| = 3$, then $\C$ pushes $x_2$, forcing $\R$ to move back to $x$, and since $|N^+(x)| =1$ (since effectively only $x_2$ has been pushed and $x_2\in N^+(x)$), $\R$ will be trapped using \cref{L:trivial}. 

                \item $x_1\neq y_1$: First, we show that if $|N^+(x_1)| = 0$, then we can trap $\R$. $\C$ begins with pushing $y$, forcing $\R$ to move to $x$. Next, $\C$ pushes $x_2$, forcing $\R$ to move to $x_1$. At this point, observe that at most one out-neighbor $x_2$ of $x_1$ is pushed. Thus, $|N^+(x_1)|=1$, and hence $\R$ will be trapped due to \cref{L:trivial}. Hence, for the rest of this case, let us assume that $N^+(x_1) = \set{w_1}$. Now, we have three cases depending on whether (i) $|N^+(x_2)| =1$, (ii) $|N^+(x_2)| = 2$, or (iii) $|N^+(x_2)|=3$.

                \smallskip
                (i) First, suppose $|N^+(x_2)| = 1$ and  $N^+(x_2) = \set{x_1}$. Now, $\C$ pushes $y$ forcing $\R$ to move to $x$. Next, $\C$ pushes $y$ again, returning the graph to its original configuration. Since both $|N^+(x_1)| = |N^+(x_2)| =1$, if $\C$ moves to any of $x_1,x_2$, it will be trapped due to \cref{L:trivial}. If it stays on $x$, then $\C$ pushes $x_1$, forcing $\R$ to move to $x_2$, which is a trap vertex since the only out-neighbor of $x_2$ (i.e., $x_1$) is pushed exactly once. Hence $\R$ will be trapped in this case. 

                \smallskip
                (ii) Second, suppose $|N^+(x_2)| = 2$, $N^+(x_2) = \set{x_1,w_2}$ and $N^-(x_2) = \set{x,w_3}$ (possibly $w_2$ or $w_3$ can be an in-neighbor or out-neighbor of $x_1$ or in-neighbor of $x$). Again $\C$ begins with pushing $y$, forcing $\R$ to move to $x$. Now, consider $N^+(x_2)$ at this point of time. If $|N^+(x_2)| = 3$ (i.e., $y=w_3$), then $\C$ pushes $x_2$, forcing $\R$ to move to $x_1$. At this point $N^+(x_1)=\set{w_1,x_2}$. We push $w_1$, forcing $\R$ to move to $x_2$. If $w_1\neq w_2$, then observe that $|N^+(x_2)| =1$ at this point and hence, $\R$ will be trapped using \cref{L:trivial}. Else, $\C$ pushes $w_1$, forcing $\R$ to move to $x$, which has at most one out-neighbor at this point, and hence $\R$ will be trapped due to \cref{L:trivial}. If $|N^+(x_2)| \leq 2$ (after $\C$ has pushed $y$), then $\C$ pushes $x_1$, forcing $\R$ to move to $x_2$, and since $|N^+(x_2)| \leq 1$ after $\C$ pushed $x_1$, $\R$ will be trapped at $x_2$ due to \cref{L:trivial}.

                \smallskip
                (iii) Third, suppose $N^+(x_2) = 1$. $\C$ begins with pushing $y$, forcing $\R$ to move to $x$. At this point $|N^+(x_2)| \leq 2$. Next, $\C$ pushes $x_1$, forcing $\R$ to move to $x_2$ and ensuring that $|N^+(x_2)|\leq 1$ again. Hence, $\R$ will be trapped at $x_2$ due to \cref{L:trivial}. 

              \end{enumerate}

        \end{enumerate}
    This completes the proof of Case~1.

    \medskip
    \noindent \textbf{Case~2.} $x_1x_2\notin E(G)$ and $y_1y_2 \notin E(G)$: Here, we will distinguish the following two cases:
    \begin{enumerate}
        \item $x_1 \notin \set{y_1,y_2}$ (i.e., $x_1 \notin N^+(y)$): Here, $\C$ begins with pushing $y$, forcing $\R$ to move to $x$. Next, $\R$ pushes $x_2$, forcing $\R$ to move to $x_1$. Since $x_1$ was a visited vertex, $|N^+(x_1)| \leq 1$ before $\C$ started pushing vertices and since none of the pushed vertices, i.e, $y,x_2$ are neighbors of $x_1$, we have that $|N^+(x_1)|\leq 1$ when $\R$ moved to $x_1$. Hence, $\R$ will be trapped by $\C$ using \cref{L:trivial}.

        \item $x_1\in \set{y_1,y_2}$: Here, since both $x_1$ and $y_1$ are visited vertices and $x_1, y_1\in N^+(y)$, we can assume without loss of generality that $x_1=y_1$. First, we establish that if $|N^+(x_2)| \neq 2$, then $\C$ can trap $\R$. Since $x\in N^-(x_2)$, if $|N^+(x_2)| \neq 2$, then either $|N^+(x_2)| \leq 1$ or $|N^+(x_2)| = 3$. In this case (i.e., when $|N^+(x_2)| \neq 2$), $\C$ begins with pushing $y$, forcing $\R$ to move to $x$. Next, $\C$ pushes the vertex $y$ again, returning the graph to its initial configuration. Now, $\R$ can either (i) stay at $x$, (ii) move to $x_1$, or (iii) move to $x_2$. (i) If $\R$ stays on $x$, then $\C$ pushes $x_2$, forcing $\R$ to move to $x_1$, and since $|N^+(x_1)| \leq 1$ at this point of time, $\R$ will be trapped at $x_1$ due to \cref{L:trivial}. (ii) If $\R$ moves to $x_1$, then $\R$ will be trapped via an application of \cref{L:trivial} since $|N^+(x_1)|\leq 1$. (iii) If $\R$ moves to $x_2$, then either $|N^+(x_2)| \leq 1$, or $|N^+(x_2)| = 3$. If $|N^+(x_2)|\leq 1$, then $\R$ gets trapped via an application of \cref{L:trivial}. If $|N^+(x_2)| = 3$, then $\C$ pushes $x_2$ to make $N^+(x_2) = \{x\}$, forcing $\R$ to move to $x$. Observe that $|N^+(x)|\leq 1$  at this point since, before $\C$ pushed $x_2$, we had $|N^+(x)| =2$ and $x_2\in N^+(x)$. Hence, $\R$ will be trapped via an application of \cref{L:trivial} at $x$. Notice, that, via symmetry, we can also conclude that if $|N^+(y_2)| \neq 2$ (given $x_1=y_1$), then $\C$ can trap $\R$. 

        Therefore, for the rest of the proof of this case, we will assume that $|N^+(x_2)| = 2$ and $|N^+(y_2)| = 2$. Let $N^+(x_2) = \set{w_3,w_2}$. Now, we again distinguish the following two cases: 
        \begin{enumerate}
            \item $y_2= x_2$: In this case, $N^+(x_2)= \set{w_3,w_2}$ and $N^-(x_2) = \set{x,y}$. $\C$ begins with pushing the vertex $y$, forcing $\R$ to move to $x$. At this point observe that $N^+(x_2) = \set{w_3,w_2,y}$. Now, $\C$ pushes $x_3$, forcing $\R$ to move to $x_2$. Now, $\C$ pushes $x_2$ ensuring that $N^+(x_2) = \set{x}$, forcing $\R$ to move to $x$. Finally, observe that since we have pushed both out-neighbors of $x$ ($x_3$ and $x_2$) exactly once and have not pushed any in-neighbor of $y$, observe that $\R$ gets trapped at $x$ as $N^+(x) = \emptyset$ at this point. 

            \item $x_2 \neq y_2$: Since $N^+(y) = \set{y_1,y_2}$ and $x_2$ is distinct from both $y_1,y_2$, we have that $\overrightarrow{yx_2} \notin E(\overrightarrow{G})$ when $\R$ moves to $u$ from $v$. Hence the only two possibilities are either $yx_2 \notin E(G)$ or $\overrightarrow{x_2y} \in E(G)$ (when $\R$ moved from $v$ to $u$). We will consider both of these possibilities separately.

            First, let $\overrightarrow{x_2y}\in E(\overrightarrow{G})$ when $\R$ moved to $u$ from $v$. $\C$ begins with pushing $y$, forcing $\R$ to move $x$. When $\R$ moves to $x$, observe that since $y$ was an out-neighbor of $x_2$ and $|N^+(x_2)| = 2$ before we pushed $y$, we have that $|N^+(x_2)| = 1$ when $\R$ reaches $x$. Next, $\C$ pushes $x_1$, forcing $\R$ to move to $x_2$. Since $x_1x_2\notin E(G)$, we have that $|N^+(x_2)| = 1$ when $\R$ moves to $x_2$, and hence $\R$ will be trapped at $x_2$ due to \cref{L:trivial}.

             \begin{figure}
                \centering
                \includegraphics[width=0.5\linewidth] {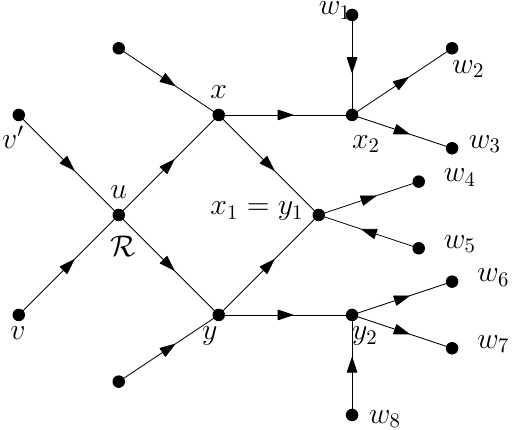}
                \caption{Illustration for the Case~2(2)(b) of \cref{C:nonEdge} when $yx_2\notin E(G)$.}
                \label{fig:final}
            \end{figure}

            Finally, we consider the  case when $yx_2 \notin E(G)$. See \cref{fig:final} for an illustration. Recall that $N^+(u) = \set{x,y}$ and let $N^-(u) = \set{v,v'}$. Similarly, let $N^+(x) = \set{x_1,x_2}$, $N^-(x) = \set{u,x'}$, $N^+(y)=\set{y_1,y_2}, N^-(y)=\set{u,y'}, N^+(x_2)=\set{w_2,w_3}, N^-(x_2)=\set{x,w_1}, N^+(x_1)=\set{w_4}, N^-(x_1)= \set{x,y,w_5}, N^+(y_2)= \set{w_6,w_7}, N^-(y_2) = \set{w_8,y}$. Here we assumed that $|N^+(x_1)| = 1$. If $N^+(x_1) = \emptyset$, then $\C$ can push $y$, forcing $\R$ to move to $x$, and then push $x_2$, forcing $\R$ to move to $x_1$, at which point \cref{L:trivial} applies. 
            
            Let $U = \set{v',v,u,x',x,y',y,x_2,x_1,y_2,w_1,\ldots,w_8}$. We note that it is possible that all vertices in $\set{v',v,u,x',x,y',y,x_2,x_1,y_2,w_1,\ldots,w_8}$ are not distinct. First, we argue that if no vertex of $U$ is pushed and $\R$ moves to either of $x$ or $y$, it will be trapped in at most two  cop moves. To see this, let $\R$ move to $x$ (resp. $y$). Then, $\C$ pushes $x_2$ (resp. $y_2$), which forces $\R$ to move to $x_1$ and since $N^+(x_1) =1$, $\R$ will be trapped in the next cop move by pushing $w_4$. 
            
            Now, as long as $\R$ stays at the vertex $u$ (i.e., it passes its moves by staying at the same vertex), $\C$ does the following: Let $\C$ occupy a vertex $z$. $\C$ finds a shortest path in $G$ from $z$ to a vertex in $U$ and either moves towards it if the orientation allows or pushes its current vertex so that the orientation in the next cop move allows the cop to move towards $U$. Using this strategy, $\C$ will reach a vertex of $U$ without ever pushing any vertex of  $U$. During all these moves, if $\R$ moves from $u$, it will be trapped. Hence, $\R$ is still at $u$. Now, consider the possible vertices where $\C$ arrives in $U$. If it arrives at either of $v,v'$, then observe that $\R$ will have to move in the next round, else it will be captured and since we have not pushed any vertex in $U$, $\R$ will be trapped in at most two rounds. If $\C$ reaches $x'$ (resp. $y'$), then $\C$ will push $y$ (resp. $x$), forcing $\R$ to move to $x$ (resp. $y$) in the next round, where it will be captured in the next round. Next, if $\C$ reaches a vertex in $\set{w_4,w_5}$, then $\C$ pushes $y$ to force $\R$ to move to $x$, and then push $x_2$ to force $\R$ to $x_1$. If $\C$ is at $w_5$, then $\R$ will be captured, else, if $\C$ is at $w_4$, then $\C$ pushes $y$, and if $\R$ moves to $w_4$ it gets captured, else $\C$ pushes $w_4$ to trap $\R$. Finally, if $\C$ reached a vertex in $\set{w_1,w_2,w_3}$ (resp. in $\{w_6,w_7,w_8\}$), then $\C$ pushes $y$ (resp. $x$). This forces $\R$ to move to $x$ (resp. $y$). Next, $\C$ pushes $x_1$, forcing $\R$ to move to $x_2$ (resp. $y_2$). Now, if $\C$ was at $w_1$ (resp. $w_8$), $\R$ would be captured in this round. Else, without loss of generality, let us assume that $\C$ is at $w_2$ (resp. $w_6$). In this case, $\C$ pushes $w_3$ (resp. $w_7$). Now, if $\R$ does not move to $w_2$ (resp. $w_6$) in the next round, it will be trapped by pushing $w_2$ (resp. $w_6$), and if it moves, observe that it will be captured by $\C$. This completes our proof for this case.

        \end{enumerate}

    \end{enumerate}

    The proof of the claim is completed by the above two exhaustive cases.
   \end{proofofclaim}
        
   The proof of our lemma follows from \cref{C:edge} and \cref{C:nonEdge}.    
\end{proof} 

Now, we present the main result of this section.
\Tregular*
\begin{proof}
    The cop $\C$ follows the strategy from \cref{L:regular} to ensure  the invariant that the out-degree of every visited vertex is at most one. If this invariant breaks, then observe that $\C$ traps $\R$ using \cref{L:regular}. Since $\overrightarrow{G}$ is finite, after a finite number of rounds, $\R$ will again visit a visited vertex where it will be trapped using \cref{L:trivial}. Finally, $\C$ can capture the trapped robber using \cref{P:trap}. This completes our proof.
    \end{proof}

Observe that \cref{T:4-regular} along with \cref{T:3-degenrate} and \cref{P:degenerate},  implies the following theorem.

\Tmaxdegree*
\section{Conclusion}\label{S:conclude}
We established that if $\overrightarrow{G}$ is an orientation of a graph with maximum degree $4$, then $\csp{G} = 1$. We also observed that if $\overrightarrow{G}$ can be transformed into a DAG using the push operation, then $\csp{G} = 1$. Huang, MacGillivray, and Yeo~\cite{huang2002pushing} showed that a chordal graph can be made a DAG via push operations if and only if it does not contain the subdigraphs depicted in \cref{fig:chordal}. Consequently, the class of chordal digraphs excluding these subdigraphs is cop-win when the cop has strong push ability. Moreover, Das et al.~\cite{das2023cops} proved that if $\overrightarrow{G}$ is an orientation of an interval graph, then $\csp{G} = 1$. This naturally raises the open question of whether all orientations of all chordal graphs are cop-win under strong push ability. We  suspect this to be the case. 
\begin{figure}
    \centering
    \includegraphics[width=0.75\linewidth]{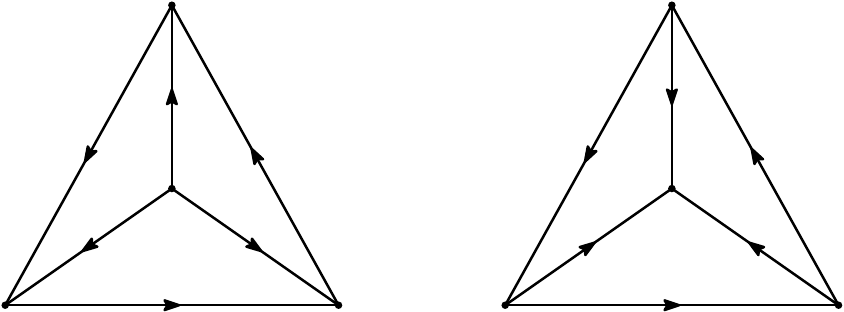}
    \caption{A chordal digraph can be pushed to be a DAG iff it does not contain any of the two orientations of $K_4$ as subdigraphs.}
    \label{fig:chordal}
\end{figure}

As discussed in \cref{S:intro}, triangle-free planar graphs are strong push cop-win. One may wonder if all planar graphs are cop-win when the cop has the strong push ability. Another interesting direction is to extend our result to $4$-degenerate graphs or graphs with maximum degree $5$.

While a broad class of graphs is shown to be strong push cop-win, we currently lack an explicit example of a graph $\overrightarrow{G}$ with $\csp{G} > 1$. Thus, a compelling direction for future work is to construct such graphs or, in a less likely but intriguing scenario, prove that every graph is cop-win under this model.

\section{Acknowledgments}
This research was supported by the IDEX-ISITE initiative CAP 20-25 (ANR-16-IDEX-0001), the International Research Center ``Innovation Transportation and Production Systems'' of the I-SITE CAP 20-25, and the ANR project GRALMECO (ANR-21-CE48-0004). We thank the anonymous reviewers for their thorough and constructive feedback, which significantly improved both the presentation and the technical clarity of this paper.
\bibliographystyle{alpha}
\bibliography{main}

\end{document}